\documentclass[reqno,12pt,a4paper]{amsart}
\usepackage[latin1]{inputenc}
\usepackage{pgfplots}
\usepackage[english]{babel}
\usepackage{amsmath,amsthm,amssymb,amsfonts, mathdots, framed}
\usepackage{scrtime, cancel}
\usepackage{enumitem, color, cases}
\usepackage{calrsfs}
\DeclareMathAlphabet{\pazocal}{OMS}{zplm}{m}{n}

\usepackage{mathtools}
\usepackage{ifpdf}
\ifpdf
\else
\usepackage[hypertex]{hyperref}
\fi
\usepackage{pdfsync,verbatim}
\usepackage{color}
\mathtoolsset{showonlyrefs}

\numberwithin{equation}{section}   

\allowdisplaybreaks
\newtheorem{theorem}{Theorem}[section]
\newtheorem{lemma}[theorem]{Lemma}
\newtheorem{proposition}[theorem]{Proposition}

\theoremstyle{definition}
\newtheorem{remark}[theorem]{Remark}

\usepackage{graphicx}

\makeatletter
\newcommand*\bigcdot{\mathpalette\bigcdot@{.5}}
\newcommand*\bigcdot@[2]{\mathbin{\vcenter{\hbox{\scalebox{#2}{$\m@th#1\bullet$}}}}}
\makeatother

\DeclareMathOperator{\tr}{tr}             
\newcount\dotcnt\newdimen\deltay
\def\Ddot#1#2(#3,#4,#5,#6){\deltay=#6\setbox1=\hbox to0pt{\smash{\dotcnt=1
\kern#3\loop\raise\dotcnt\deltay\hbox to0pt{\hss#2}\kern#5\ifnum\dotcnt<#1
\advance\dotcnt 1\repeat}\hss}\setbox2=\vtop{\box1}\ht2=#4\box2}

\def\Blue{\color{blue}}

\usepackage{amsmath, color}
\usepackage{latexsym}
\usepackage{amssymb}

\newcommand{\R}{\mathbb{R}}
\newcommand{\C}{\mathbb{C}}
\newcommand{\N}{\mathbb{N}}

\title[On the orthogonality of eigenspaces]{
On the orthogonality of generalized eigenspaces for the  Ornstein--Uhlenbeck  operator }
\author{Valentina Casarino}
\address{Universit\`a degli Studi di Padova\\Stradella san Nicola 3 \\I-36100 Vicenza \\ Italy}
\email{valentina.casarino@unipd.it}
\author{Paolo Ciatti}
\address{Universit\`a degli Studi di Padova\\Via Marzolo 9 \\I-35100 Padova \\ Italy}
\email{paolo.ciatti@unipd.it}
\author{Peter Sj\"ogren}
\address{Mathematical Sciences,  University of Gothenburg and  Mathematical Sciences\\ \hspace*{9pt} Chalmers University of Technology  \\ SE - 412 96 G\"oteborg, Sweden}
\email{peters@chalmers.se}

\thanks{The first and  second authors
are members of the Gruppo Nazionale per l'Analisi Matematica, la Probabilit\`a e le loro Applicazioni (GNAMPA) of the Istituto Nazionale di Alta Matematica (INdAM)
and
were partially supported by GNAMPA (Project 2020
``Alla frontiera tra l'analisi complessa in pi\`u variabili e l'analisi armonica") }
\date{\today, \thistime}

\begin{document}

\begin{abstract}
We study  the orthogonality  of the generalized eigenspaces of an  Ornstein--Uhlenbeck operator $\mathcal L$ in $\R^N$,
with drift given by a real   matrix $B$ whose eigenvalues have negative real parts.
 If $B$ has only one eigenvalue, we prove that  any two distinct generalized eigenspaces  of
 $\mathcal L$ are orthogonal with respect to the invariant Gaussian measure.
Then we show by means of two examples that if $B$ admits distinct eigenvalues, the generalized eigenspaces of $\mathcal L$ may
or may not be orthogonal.
\end{abstract}


\maketitle
\section{Introduction}

In this note we discuss the orthogonality  of the generalized eigenspaces associated to a general Ornstein--Uhlenbeck operator $\mathcal L$ in $\R^N$.

Recently, the authors started studying some harmonic analysis  issues
 in a nonsymmmetric Gaussian  context \cite{CCS1, CCS2, CCS3}.
In particular,  the Ornstein--Uhlenbeck  semigroup $\big(\mathcal H_t\big)_{t>0}$  generated by $\mathcal L$
 is  not assumed self-adjoint  in $L^2(\gamma_\infty)$; here
 $\gamma_\infty$  denotes the unique invariant probability measure under the action of  the semigroup, and  will be specified later.

In this general framework,
the  Ornstein--Uhlenbeck operator $\mathcal L$
 admits a complete system of generalized eigenfunctions; see \cite{MPP}. But without
 self-adjointness, the orthogonality of distinct  eigenspaces of $\mathcal L$ is  not guaranteed.
In fact, while the kernel of $\mathcal L$ is always orthogonal to the other generalized eigenspaces of $\mathcal L$
 in $L^2(\gamma_\infty)$,
 the question of orthogonality
between generalized eigenspaces associated to nonzero eigenvalues is more delicate.
As expected, the spectral properties of $B$  play a prominent role here.
Indeed, we prove in Section  \ref{single}
 that if $B$ has a unique eigenvalue, then any two
 generalized eigenfunctions of $\mathcal L$ corresponding to different eigenvalues are orthogonal in $L^2(\gamma_\infty)$.

Then in Sections \ref{Example1} and   \ref{ad-case}
 we exhibit two examples showing, respectively, that
if $B$ admits two distinct eigenvalues,
 the generalized eigenspaces associated to $\mathcal L$
may or may not be orthogonal.
The last section also contains a result which relates orthogonality of the eigenspaces of  $\mathcal L$ to that of the eigenspaces of the drift matrix, under some restrictions.

In the following, the symbol
  $I_k$ will  denote  the identity matrix of size $k$,
  and we omit the subscript when the size is obvious.
  We will write  $\langle .,.\rangle$ for scalar products both in $\R^N$ and in
  $L^2(\gamma_\infty)$.
By $\N$ we mean $\{0, 1, \dots\}$.

\section{The Ornstein--Uhlenbeck operator}\label{OU-generalities}
In this section,
we specify  the  definition of the Ornstein--Uhlenbeck operator  $\mathcal L$ and recall  some known facts concerning its spectrum.

We consider  the
Ornstein--Uhlenbeck semigroup
$\big(
\mathcal H_t^{Q,B}
\big)_{t> 0}\,,$  given for all  bounded continuous functions $f$ in $\R^N$, $N\ge 1$, and all $t>0$
 by the Kolmogorov formula
\begin{equation*}
\mathcal H_t^{Q,B}
f(x)=
\int
f(e^{tB}x-y)\,d\gamma_t (y)\,, \quad x\in\R^N\,,
\end{equation*}
(see \cite{Kolmog}).
Here $B$ is a real $N\times N$ matrix whose eigenvalues have negative real parts, and $Q$ is a real, symmetric and positive-definite $N\times N$ matrix.
Then
we introduce  the covariance  matrices
$$Q_t=\int_0^t
e^{sB}\,Q\,
e^{sB^*}ds\,,\qquad \text{ $t\in (0,+\infty]$,
}$$
  all symmetric and positive definite.
Finally, the normalized Gaussian measures  $\gamma_t$ are defined for  $t\in (0,+\infty]$ by
$$
d\gamma_t(x) =
(2\pi)^{-\frac{N}{2}}
(\text{det} \, Q_t)^{-\frac{1}{2} }
e^{-\frac12 \langle Q_t^{-1}x,x\rangle}\,dx.
  $$
As mentioned above,  $\gamma_\infty$ is the unique invariant probability measure of the Ornstein--Uhlenbeck semigroup.

The Ornstein--Uhlenbeck operator is  the  infinitesimal generator  of the semigroup
$\big(
\mathcal H_t^{Q,B}
\big)_{t> 0}\,$, and it is explicitly given by
$$\mathcal L^{Q,B} f(x)=
\frac12
\tr
\big(
Q\nabla^2
f\big)(x)+
\langle Bx, \nabla f(x)\rangle
\,,\qquad
{\text{ $f\in
\mathcal S (\R^N)$,}}
$$
where $\nabla$ is the gradient and $\nabla^2$ the Hessian.

By convention, we abbreviate
$
\mathcal H_t^{Q,B}
$
 and
 $\mathcal L^{Q,B} $
 to $\mathcal H_t$
 and
 $\mathcal L$, respectively.
We can thus write
$\mathcal H_t=e^{t\mathcal L}$.


In \cite[Theorem 3.1]{MPP} it is verified that the spectrum of  $\mathcal L$
is the set
\begin{equation}\label{spectrum_L}
 \left\{
 \sum_{j=1}^r n_j \lambda_j\,:\, n_j\in\mathbb N    \right\},
\end{equation}
where $ \lambda_1, \ldots,  \lambda_r$ are the  eigenvalues of  the drift matrix $B$.
 In particular,
 $0$ is an  eigenvalue of  $\mathcal L$, and the corresponding eigenspace $\ker  \mathcal L$ is one-dimensional and consists of all  constant functions,
 as proved in \cite[Section 3]{MPP}.

We also recall that, given a linear operator $T$ on some $L^2$ space, a number $\lambda \in \C$   
is a  generalized eigenvalue  of $T$ if there exists a nonzero $u \in L^2$ such that
$(T - \lambda I)^k \,u=0$
for some positive integer $k$. Then $u$
 is called a generalized eigenfunction, and those $u$ span the  generalized eigenspace corresponding  to $\lambda$.
 As already recalled,
 it is known from \cite[Section 3]{MPP} that
the  Ornstein--Uhlenbeck operator $\mathcal L$
 admits a complete system of generalized eigenfunctions, that is, the linear span
 of the generalized eigenfunctions  is dense in $L^2(\gamma_\infty)$.

\vskip7pt

\noindent \textit{Subsection 2.1. \hspace{-6pt} Use of Hermite polynomials}.\hskip7pt
As proved in \cite{MPRS},
a suitable linear change of coordinates in $\R^N$ makes  $Q=I$ and $Q_\infty$ diagonal.
  When applying this,
  we adhere  to the notation introduced in  Lemma 1 in  \cite{Chill}, where also the following facts can be found.
Let $\mathbf H_n$ denote the  space of Hermite polynomials of degree $n$ in these coordinates, adapted
 by means of a dilation
  to $\gamma_\infty$
in the sense that the  $\mathbf H_n$ are mutually orthogonal in  $L^2(\gamma_\infty)$ (they are written $H_{\lambda,k}$ in \cite{Chill}).
The classical Hermite expansion 
 (called the It\^o-Wiener decomposition in \cite{Chill}) says that $L^2(\gamma_\infty)$
is the closure of the direct sum of the  $\mathbf H_n$; we refer to  
\cite[p.\,64]{Wiener} for a proof in dimension one and note that  the extension to higher dimension  is  trivial.
In other words, we can decompose  any function  $u\in L^2(\gamma_\infty)$
  as 
   \begin{equation}\label{decomp}
  u = \sum_j u_j
  \end{equation}
with $u_j \in \mathbf H_j$
and convergence in $L^2(\gamma_\infty)$.
Further, each
$\mathbf H_n$  is
 invariant under  $\mathcal L$; see \cite[Proposition 1]{Chill}.

The 
Hermite decomposition  implies, in particular, that each  generalized eigenfunction of $\mathcal L$
with a nonzero
eigenvalue is
 orthogonal to the space of constant functions, that is, to the kernel of $\mathcal L$.
Anyway, we provide here a  proof of this fact which is independent of Hermite polynomials.

\begin{lemma}
 Let $\lambda \ne 0$.
 If $u \in L^2(\gamma_\infty)$  and $(\mathcal{L} - \lambda)^k \,u = 0$ for some $k \in \{ 1, 2,\dots\}$,
 then $\int u\,d\gamma_\infty = 0$.
\end{lemma}

\begin{proof}
The implication is trivial if we set  $k=0$, so assume it holds for some $k \ge 0$ and that
$(\mathcal{L} - \lambda)^{k+1}\, u = 0$.

Then
\begin{equation*}
  \mathcal{L}(\mathcal{L} - \lambda)^k\, u   = \lambda   (\mathcal{L} - \lambda)^k\, u,
\end{equation*}
and thus for any $t>0$
\begin{equation*}
 e^{t \mathcal{L}}(\mathcal{L} - \lambda)^k \,u   = e^{t\lambda}   (\mathcal{L} - \lambda)^k\, u.
\end{equation*}
These operators commute, so
\begin{equation*}
(\mathcal{L} - \lambda)^k\, e^{t \mathcal{L}} u   =(\mathcal{L} - \lambda)^k \,e^{t\lambda}  u,
\end{equation*}
that is,
\begin{equation*}
(\mathcal{L} - \lambda)^k \,
(e^{t \mathcal{L}} u - e^{t\lambda} u) = 0.
\end{equation*}
The induction assumption now implies that
\begin{equation*}
\int (e^{t \mathcal{L}} u - e^{t\lambda} u)\,d\gamma_\infty = 0.
\end{equation*}
Since $\gamma_\infty$ is invariant under the semigroup, this means that
\begin{equation*}
\int u \,d\gamma_\infty  =   e^{t\lambda} \int u \,d\gamma_\infty
\end{equation*}
for all $t>0$. Thus the integral vanishes.
\end{proof}
\bigskip

{\Blue{



}}

\section{The case when $B$ has only one eigenvalue}\label{single}

\begin{proposition}\label{oneeigenvalue}
If the drift matrix $B$ has only one eigenvalue, then
any two generalized eigenfunctions of $\mathcal L$ with different eigenvalues are orthogonal with respect to $\gamma_\infty$.
\end{proposition}

We let  $\lambda$ be the unique eigenvalue of $B$, which is necessarily real and negative. It is known that all generalized eigenfunctions of  $\mathcal L$ are polynomials, see
 \cite[Thm. 9.3.20]{Lorenzi}.

We first state a lemma and use it to prove the proposition.

\begin{lemma}\label{degreen}
Let  $u$ be a generalized eigenfunction of  $\mathcal L$ which is a polynomial of degree $n \ge 0$. Then  the corresponding eigenvalue is $n\lambda$.
\end{lemma}

\noindent \textit{Proof of Proposition \ref{oneeigenvalue}}. \hskip4pt
 Let  $u$ be  a generalized eigenfunction of  $\mathcal L$, thus satisfying $(\mathcal L - \mu)^k\, u = 0$ for some $\mu \in \C$ and $k \in \N$.
 Applying the coordinates from Subsection 2.1, we can decompose  $u$
 as in \eqref{decomp}, where the sum is now finite.
   Since then
 \begin{equation*}
  \sum_j (\mathcal L - \mu)^k u_j  = 0
 \end{equation*}
 and each term here is in the corresponding $ \mathbf H_j$, all the terms are 0.
 But this is compatible with  Lemma \ref{degreen} only if there is only one nonzero term in the decomposition of $u$. Thus $u  \in \mathbf H_n$, where $n$ is the polynomial  degree of $u$.

Lemma  \ref{degreen} then implies that two generalized eigenfunctions with different eigenvalues
are of different degrees and thus belong to
 different  $\mathbf H_n$. The desired orthogonality now follows from that of the $\mathbf H_n$.
\qed                   

\medskip

\noindent {\it{Proof of  Lemma  \ref{degreen}.}} \hskip4pt
Let $u$ be  a generalized eigenfunction of  $\mathcal L$ of polynomial  degree $n$. We denote the corresponding eigenvalue by $\mu$.
Decomposing  $u$ as in \eqref{decomp}, we see that this sum is for $j\le n$
and that the term  $u_n$ is nonzero and a  generalized eigenfunction of  $\mathcal L$ with eigenvalue $\mu$.
For some $m$, the function  $(\mathcal L - \mu)^m u_n$ will then be an
eigenfunction   with the same eigenvalue. This function is in  $ \mathbf H_n$
and thus a polynomial of degree $n$. As a result,
 we can assume that $u$ is actually an eigenfunction of  $\mathcal L$,
when proving the lemma.

 We now choose coordinates in $\R^N$ that give a Jordan decomposition of $B$. This means that $B = \lambda I + R$, where $R = (R_{i,j})$ is a matrix with nonzero entries only in the first subdiagonal. More precisely, $R_{i,i-1} = 1$ for $i \in P$, where P is a subset of $\{2,\dots,N\}$,
and all other entries of $R$ vanish.

We write  $\mathcal L = \mathcal S+\mathcal B$, where
\begin{equation*}
 \mathcal B f(x)  = \langle Bx, \nabla f(x)\rangle,
 \end{equation*}
 and $\mathcal S$ is the remaining, second-degree part of $\mathcal L$.
 Notice that, when applied to polynomials, $\mathcal B$ preserves the degree whereas $\mathcal S$ decreases it by 2. So if $v$ is the $n$th-degree part of $u$,  we must have $\mathcal B v = \mu v$.

We let $\mathcal B$ act on a monomial $x^\alpha$, where $\alpha \in \N^N$ is a multiindex, getting
\begin{align*}
\mathcal  B     x^\alpha &= \sum_j \lambda x_j\, \frac {\partial x^\alpha}{\partial x_j}   +      \sum_{i \in P} x_{i-1}\, \frac {\partial x^\alpha} {\partial x_i}    \\
  &=  \lambda  \sum_j \alpha_j \,x^\alpha   +    \sum_{i \in P} \alpha_i \,\frac {x_{i-1} }{x_i} \,x^\alpha
  = \lambda n\, x^\alpha        +    \sum_{i \in P}     \alpha_i \,        x^{\alpha^{(i)}},\end{align*}
where $\alpha^{(i)} = \alpha + e_{i-1} - e_i$ for $i \in P$.  Here $\{e_{j}\}_{j=1}^n$
denotes
the standard basis in $\R^N$.
 Thus the restriction
 of  $\mathcal B$ to the space of homogeneous polynomials of degree $n$ is given as $\lambda n I + \mathcal R$, where $\mathcal R$ is the linear operator that maps  $x^\alpha$ to $  \sum_{i \in P}    \alpha_i  \,     x^{\alpha^{(i)}}$.

 We claim that  the only eigenvalue of  $\mathcal R$ is 0.
If so, the only eigenvalue of the restriction of $\mathcal B$ mentioned above is $\lambda n$, which would  prove the lemma since  $\mathcal B v = \mu v$.

In order to  prove this claim,
we define for any  $\alpha \in \N^N$ with $|\alpha| = n$
 \begin{equation*}
 V(\alpha) = \sum_1^N j \alpha_j.
 \end{equation*}
Clearly  $V(\alpha^{(i)}) = V(\alpha) - 1$.
We  select a basis in the linear space of all homogeneous polynomials of degree $n$ consisting of all monomials $x^\alpha$ with $|\alpha| = n$, enumerated
in such a way that $V$ is nondecreasing. The definition of $\mathcal R$ now shows that its matrix with respect to this basis is upper triangular
 with
zeros on the diagonal. The claim follows, and so does  the lemma.
 \qed
\bigskip

\section{$B$  has two distinct eigenvalues:
a first example
}\label{Example1}

The following example shows that the generalized eigenspaces of the Ornstein--Uhlenbeck operator may be orthogonal even in the case when $B$ has more than one
eigenvalue.

In two dimensions, we let
\begin{equation}\label{defQBB}
Q=I_2 \;\;\text{ and }   \;\;B =
\begin{pmatrix}
 -1 & 1 \\
  -1 & -1
\end{pmatrix}.
\end{equation}
whose eigenvalues are $-1 \pm i$.

\vskip7pt

\begin{proposition}
With $N=2$, let $Q$ and $B$ be  as in \eqref{defQBB}.
 Then each generalized eigenfunction of
$\mathcal L$ is an eigenfunction. Moreover, any two eigenfunctions of
$\mathcal L$ with different eigenvalues are  orthogonal with respect to
$\gamma_\infty$.
\end{proposition}
\begin{proof}
One finds that
\begin{equation*}
e^{sB} = e^{-s}
\begin{pmatrix}
  \cos s & \sin s \\
  -\sin s & \cos s
\end{pmatrix}
\end{equation*}
and
\begin{equation*}
e^{sB}\,e^{sB^*} = e^{-2s}\,I_2,
\end{equation*}
so that
\begin{equation*}
Q_\infty = \frac12\,I_2,  \qquad \qquad Q_\infty^{-1} = 2\,I_2.
\end{equation*}
The invariant measure is therefore
\begin{equation}\label{invm}
d\gamma_\infty(x) = \pi^{-1}\,\exp\left(-x_1^2  -x_2^2\right)\,dx.
\end{equation}

Since $Q=I_2$ and $Q_\infty$ is diagonal,  we are in the situation treated in
\cite{Chill}; see also Subsection 2.1. But  \cite{Chill}
defines the analog of our $\mathcal L$ (denoted $A$) without the coefficient
$1/2$ in front of the second-order term $\Delta$. We will adapt to the definitions and notation of
\cite{Chill}. Therefore, we replace  $\mathcal L$ by
\begin{equation*}
A = 2\mathcal L = \Delta + \langle \tilde Bx,\nabla \rangle,
\end{equation*}
where $\tilde B = 2B$.  Obviously, $A$ and $\mathcal L$ have the same (generalized) eigenfunctions.

From \cite{Chill} we will also need the diagonal matrix denoted $D_\lambda$,
which in the case considered is                         
\begin{equation*}
D_\lambda =
 \begin{pmatrix}
  \lambda_1 & 0 \\
  0 & \lambda_2
\end{pmatrix}
\end{equation*}
with $\lambda_1  = \lambda_2 = 1/2$.

The invariant measure is the same for the semigroups $\left(e^{tA}\right)$ and
$\left(e^{t\mathcal L}\right)$ by uniqueness, and given by \eqref{invm}.

For $k = (k_1,k_2) \in \N^2$ we introduce two-dimensional Hermite polynomials
\begin{equation*}
H_k(x_1,x_2) = H_{k_1}(x_1) \, H_{k_2}(x_2),
\end{equation*}
where  $H_{k_i}$ is the classical Hermite polynomial.
Here we should point out that  \cite{Chill} uses dilated Hermite polynomials denoted $H_{\lambda,k}$. But in our case there are no dilations, since the dilation factors are $\sqrt{2\lambda_i} = 1$, \hskip2pt $i=1,2$.

 The $H_k$ are mutually
orthogonal in $L^2(\gamma_\infty)$,  and
\begin{equation*}
\widetilde H_{k_1,k_2} = \frac1 {\sqrt{2^{k_1+k_2}\,k_1!\,k_2!}}\, H_{k_1,k_2}
\end{equation*}
are orthonormal. The space
 $\mathbf H_n$ is generated by the  Hermite polynomials  $H_k$ of degree $k_1+k_2 = n$, for  $n \in \N$, and  $\mathbf H_n$ is invariant under $A$.

 As in \cite[Section 4]{Chill}, we split the partial differential operator $A$ as
 $A = A_1 + L$ with
$A_1 = \Delta - \langle D_\lambda^{-1}{{x}},\, \nabla\rangle$ and
$ L = \langle C x,\, \nabla \rangle $,
 where $x=(x_1,x_2)$ and
$C = (c_{i,j})$ is the skew-symmetric matrix
 \begin{equation*}
C  = \tilde B + D_\lambda^{-1} =
\begin{pmatrix}
  0 & 2 \\
  -2 & 0
\end{pmatrix}.
\end{equation*}

Then each polynomial in  $\mathbf H_n$ is an eigenfunction of $A_1$ with eigenvalue $-2n$.
This follows from the differential equation satisfied by the Hermite polynomials; see \cite[formula (3.5)]{Chill}.

The effect of $L$ on Hermite polynomials can be read off from
a formula on page 711 of \cite{Chill}, best described as that following the
words "we find that".
We adapt this formula to the  normalized polynomials $\widetilde H_{k_1,k_2}$ with $k_1+k_2 = n$, and choose $\ell_i = k_i \pm 1$. Here the
scalar product written $\langle.,.\rangle$ is taken  in $L^2(\gamma_\infty)$.
We also use the trivial observation that $c_{1,2}\,\lambda_2 = 1$ and
$c_{2,1}\,\lambda_1 = -1$.

The result is
\begin{align*}
\langle L\,& \widetilde H_{k_1,k_2}, \widetilde H_{k_1-1,k_2+1}\rangle \\&=
\frac1 {\sqrt{2^{k_1+k_2}\,k_1!\,k_2!}}\,
\frac1 {\sqrt{2^{k_1+k_2}\,(k_1-1)!\,(k_2+1)!}}
 \,2^{k_1+k_2+1}\,k_1!\,(k_2+1)!\\
&= 2 \,\sqrt{k_1}\,\sqrt{k_2+1}
\end{align*}
for $1\le k_1 \le n$ and $0\le k_2 \le n-1$;
similarly
\begin{align*}
\langle L\, & \widetilde H_{k_1,k_2}, \widetilde H_{k_1+1,k_2-1}\rangle \\&= -\frac{1}{\sqrt{2^{k_1+k_2}\,k_1!\,k_2!}}\,
\frac{1}{\sqrt{2^{k_1+k_2}\,(k_1+1)!\,(k_2-1)!}}
 \,2^{k_1+k_2+1}\,(k_1+1)!\,k_2!\\
&= -2 \,\sqrt{k_1+1}\,\sqrt{k_2}
\end{align*}
for $0\le k_1 \le n-1$ and $1\le k_2 \le n$.
This describes the restriction of $L$  to $\mathbf H_n$  completely, since
\begin{equation*}
\langle L\, \widetilde H_{k_1,k_2}, \widetilde H_{\ell_1,\ell_2}\rangle \ne 0
\qquad \mathrm{only\hskip7pt when} \qquad (k_1-\ell_1)(k_2-\ell_2) =-1.
\end{equation*}

 In $\mathbf H_n$ we use the orthonormal basis $ \widetilde H_{n-\kappa,\kappa}\,,\; \,\kappa = 0,1,\dots, n$. A consequence of the preceding formulas is that the matrix  $L^{(n)} = (L^{(n)}_{i,j})$ of the restriction of $L$ with respect to this basis is given by
 \begin{equation*}
L^{(n)}_{\kappa +1,\kappa} = 2\,\sqrt{\kappa +1}\,\sqrt{n-\kappa}
= - L^{(n)}_{\kappa ,\kappa+1}, \qquad        0 \le \kappa \le n-1,
\end{equation*}
 all other entries of the matrix being 0. Thus  $L^{(n)}$  is skew-symmetric. The restriction to  $\mathbf H_n$ of the operator
 $A = A_1 + L$ has matrix
 $-2nI_{n+1}+ L^{(n)}$,
 and it follows that $A$ is a normal operator in
 $\mathbf H_n$. The spectral theorem for normal operators implies that
 $A$ can be diagonalized in  $\mathbf H_n$ by means of an orthogonal change of coordinates.
  Since the spaces $\mathbf H_n, \; n \in \N$, are mutually orthogonal, this proves  the proposition.
\end{proof}

\bigskip

\section{$B$  has two distinct eigenvalues:
a second example}\label{ad-case}

  In this section we exhibit a class of drift matrices $B$ with two different eigenvalues (which, in contrast with  those in  the example in Section \ref{Example1}, are real),
but such that the generalized eigenspaces associated to the corresponding Ornstein--Uhlenbeck operator $\mathcal L$ are not orthogonal.

In $\R^2$ we consider $Q=I_2$ and
\begin{equation}\label{BBB}
 B =
\begin{pmatrix}
  -a+d&0 \\
  c&-a-d
\end{pmatrix},
\end{equation} with $a>d>0$ and $c\neq 0$.
To compute the exponential of $sB$, we write $B= -aI + M$, where
\begin{equation*}M =
\begin{pmatrix}  d&0 \\
  c&-d\end{pmatrix}.
\end{equation*}
 Since  $MM = d^2 I$,   we get for $s>0$
 \[
 \exp(sB)= e^{-as}\,\left( \cosh(sd)\, I  +  d^{-1} \, \sinh (sd)\, M \right)\,.
 \]
This leads to
\begin{align*}
\exp(sB)\,\exp(sB^*)
=& \,
  e^{-2as}
 \begin{pmatrix}
   e^{2sd}& \frac c d \, e^{sd} \sinh (sd)  \\
\frac c d \, e^{sd} \sinh (sd)\qquad
  & \frac {c^2} {d^2} \,   \sinh^2 (sd) +  e^{-2sd}
\end{pmatrix}.
\end{align*}
Integrating this matrix over $0<s<\infty$, we obtain
\begin{align*}
Q_\infty 
&=\begin{pmatrix}
\frac1{2(a-d)}& \frac{c}{4a(a-d)} \\
 \frac{c}{4a(a-d)} &  \frac{c^2}{4a(a-d)(a+d)} + \frac{1}{2(a+d)}
 \end{pmatrix}
 ,
 \end{align*}
and so
\begin{align*}
\frac12\,Q_\infty^{-1}
&=\frac1{c^2+4a^2}
\begin{pmatrix}
{2a[c^2+2a(a-d)]}{}
& -{2ac(a+d)}{} \\
{}&{}\\
-{2ac(a+d)}{} &{4a^2(a+d)}{}
 \end{pmatrix}.
\end{align*}

The invariant measure $\gamma_\infty$ is thus proportional to
\begin{multline*}
 \exp\left(-
\frac{2a[c^2+2a(a-d)]}{c^2+4a^2}\, x_1^2
 +\frac{4ac(a+d)}{c^2+4a^2}\, x_1 x_2
- \frac{4a^2(a+d)}{c^2+4a^2}\, x_2^2\right)dx
\\
= \exp\big(-
{(a-d)}\, x_1^2\big)\;
\exp\left(-\frac{a+d}{c^2+4a^2}
\left(c x_1-2a  x_2
\right)^2\right)dx.
\end{multline*}

Writing
$z_1 =\sqrt{a-d}\: x_1$ and
$z_2 =\sqrt{\frac{a+d}{c^2+4a^2}}\:
\big(2a  x_2-c x_1\big)$
and recalling that  $\gamma_\infty$ is a probability measure,
we see that                      
\begin{align*}
d\gamma_\infty&=
 \pi^{-1}\,\exp\big(- z_1^2-
z_2^2\big)\, dz.
\end{align*}

To find some eigenfunctions of $\mathcal L$, we consider
 polynomials in $x_1,x_2$ of degree 2. 
One finds that
 \begin{align*}
  v_1 &= x_1^2 -\frac1{2(a-d)},\\
   v_2 &= x_1^2 -\frac{2d}{c}\, x_1x_2 -\frac1{2a},\\
  v_3 &= x_1^2 -\frac{4d}{c}\,x_1x_2 + \frac{4d^2}{c^2}\, x_2^2 -\frac{c^2+4d^2}{2c^2(a+d)}
\end{align*}
are eigenfunctions, with  eigenvalues   $-2(a-d),\:
  -2a$ and  $-2(a+d)$, respectively.

 Any two of these polynomials turn out not to be orthogonal with respect to
the invariant measure, as follows by straightforward computations. We sketch one example.

One simply multiplies $v_1$ and $v_3$ and rewrites the product in terms of $z_1$ and $z_2$. Doing so, one can neglect all terms of odd order in  $z_1$ or $z_3$, when integrating with respect to $\gamma_\infty$. Writing "$\mathrm{ odd}$" for such terms, we find that the product is

\begin{align*}
& \frac{1 }{a^2}\,  z_1^4
 + \frac{d^2 (c^2+4a^2)}{a^2c^2 (a^2-d^2)}\,z_1^2 z_2^2
 -\Big[\frac{c^2+4d^2}{2c^2(a^2-d^2)} +\frac{1}{2a^2} \Big]\, z_1^2
\\
&-\frac{d^2(c^2+4a^2)}{2a^2c^2(a^2-d^2)}\,z_2^2
+
  \frac{c^2+4d^2}{4c^2(a^2-d^2)}
  + \mathrm{ odd}.
\end{align*}
Integrating and simplifying, we get
\begin{align*}
 \int& v_1v_3 \,d\gamma_\infty
    =
 \frac{1 }{2a^2 } \,>\,0,
\end{align*}
so $v_1$ and $v_3$ are not orthogonal.

\begin{remark}
Let now  $d=a/2$ in this example.
Then the fourth-degree polynomial
\begin{align*}
  v_4 = x_1^4 - \frac{6}{a}\,x_1^2 + \frac{3}{a^2}
\end{align*}
is an eigenfunction of $\mathcal L$ with eigenvalue $-2a$, like $v_2$.
Thus eigenfunctions of different polynomial degrees can have the same eigenvalue.
This shows that for an eigenfunction $u$, the sum in \eqref{decomp}
 may consist of more than one term, and a (generalized) eigenspace need not be contained in one   $\mathbf H_n$.
\end{remark}

\vskip5pt

  The eigenvalues of the matrix $B$  defined in  \eqref{BBB} are $-a\pm d$, and it is easily seen that the corresponding eigenspaces are not orthogonal  in $\R^2$. This turns out to be related to the non-orthogonality  of the eigenspaces of  $\mathcal L$, at least in two dimensions, in the following way.


\begin{proposition}
Let  $N=2$ and $Q = I$, and                    
assume that $B$ has two different,  real eigenvalues.
Then the generalized eigenspaces of $\mathcal L$  are orthogonal in  $L^2(\gamma_\infty)$  if and only if the two eigenspaces  of $B$ are orthogonal in $\R^2$.
  \end{proposition}
\begin{proof}

    To begin with, we consider a coordinate change $\widetilde x = Hx$, where $H$ is an orthogonal matrix.             Simple computations show that the operator $\mathcal L^{Q,B}$ is transformed to $\mathcal L^{\widetilde Q, \widetilde B}$ in the new coordinates, with $\widetilde Q =HQH^*$ and $\widetilde B =HBH^*$; cf.\
   \cite[p. 474]{MPRS}. In our case, $\widetilde Q = Q = I$. The eigenvalues of $B$ and the angle between its eigenvectors will not change.

To prove the proposition, assume first that the (real) eigenvectors of $B$ are orthogonal in $\R^2$.
 Then $B$ is symmetric, since it can be diagonalized by means of an orthogonal  change of coordinates as just described.
 This implies that $\mathcal L$
 is symmetric (\cite[Proposition 9.3.10]{Lorenzi}), so that the orthogonality of its eigenspaces is trivial.

          Next, we assume that
   the eigenvectors of $B$ are  not orthogonal in $\R^2$. 
  By Schur's decomposition theorem (see \cite[Theorem 2.3.1]{Horn}) there exists an orthogonal  change of coordinates which makes  $B$
             lower triangular, though not diagonal. We are thus in
  the situation described in  \eqref{BBB}. As we have seen,   some eigenspaces of $\mathcal L$ are then not orthogonal with respect to the invariant measure.
 \end{proof}

 We finally remark that the  ``if" part of this proposition easily extends to arbitrary dimension $N$.
 Then it is assumed that  $B$ has N different,  real eigenvalues with mutually orthogonal eigenspaces.




\begin{thebibliography}{M\"uPeRi2}

\bibitem{CCS1}
V. Casarino, P. Ciatti and P. Sj\"ogren,
The maximal operator of a normal  Ornstein-Uhlenbeck  semigroup is of  weak type (1,1),
 {\em Ann. Sc. Norm. Sup. Pisa Cl. Sci.} (5)  {\textbf{XXI}} (2020), 385-410.

\bibitem{CCS2}
 V. Casarino, P. Ciatti and P. Sj\"ogren,
{ On the maximal operator of a
general Ornstein--Uhlenbeck  semigroup},
 {\tt arXiv:1901.04823}, submitted.

\bibitem{CCS3}
 V. Casarino, P. Ciatti and P. Sj\"ogren,
{Riesz transforms of a general Ornstein--Uhlenbeck semigroup},
(2020),   {\tt arXiv:2004.04022}, {\em Calculus of Variations and Partial Differential Equations},
to appear.

\bibitem{Chill}
 R. Chill, E. Fasangova, G. Metafune and D. Pallara,
 {The sector of analyticity of the Ornstein-- Uhlenbeck semigroup on $L^p$ spaces with respect to invariant measure},
 {\em J. London Math. Soc.} (2) {\textbf{71}} (2005), 703--722.



\bibitem{Horn}
R.A. Horn and C.R. Johnson, {\em{Matrix Analysis}}, Cambridge University Press, (2012).

\bibitem{Kolmog}
A. N. Kolmogorov,
Zuf\"allige  Bewegungen,
{\em Ann. of Math.} \textbf{116} (1934), 116--117.

\bibitem{Lorenzi}
 L. Lorenzi and M. Bertoldi,
{\em Analytical methods for Markov semigroups},
Pure and Applied Mathematics (Boca Raton), 283, Chapman \& Hall/CRC, Boca Raton, FL, 2007.

\bibitem{MPP}
G.Metafune, D. Pallara and E. Priola,
Spectrum of Ornstein-Uhlenbeck operators in $L^p$ spaces with respect to invariant measures,
{\em{J.  Funct. Anal.}} \textbf{196} (2002), 40--60.

\bibitem{MPRS}
G. Metafune, J. Pr\"uss, A. Rhandi, and R. Schnaubelt,
The domain of the Ornstein-Uhlenbeck operator on a $L^p$-space with invariant measure,
{\em Ann. Sc. Norm. Super. Pisa Cl. Sci.} \textbf{1} (2002), 471--487.




\bibitem{Wiener}
N. Wiener,
{\em The Fourier Integral and Certain of its Applications},
Cambridge University Press, (1933).

\end{thebibliography}
\end{document}